\documentclass[12pt]{amsart}
\usepackage[utf8]{inputenc}
\usepackage{amssymb,amsmath,xcolor,amsrefs,amsthm,fullpage}
\usepackage{amsthm}
\usepackage{url}
\usepackage{hyperref}
\usepackage{tikz,subcaption}
\usepackage[enableskew]{youngtab}
\usepackage{ytableau,varwidth}
\usepackage[hang,flushmargin]{footmisc} 

\definecolor{red}{rgb}{1,0,0}
\definecolor{blue}{rgb}{.2,.2,.8}
\definecolor{magenta}{rgb}{1,0,1}
\definecolor{dartmouthgreen}{rgb}{0.05, 0.5, 0.06}
\definecolor{purple(x11)}{rgb}{0.63,0.36,0.94}
\definecolor{turquoise}{rgb}{0.25, 0.87, 0.81}

\newtheorem{theorem}{Theorem}[section]
\newtheorem{corollary}[theorem]{Corollary}

\newtheorem{conj}{Conjecture}
\theoremstyle{definition}

\newtheorem{example}{Example}
\newtheorem{remark}{Remark}

\title{
On the number of parts in all partitions enumerated by the Rogers-Ramanujan identities}
\author{Cristina Ballantine and  Amanda Folsom}
\date{}

\address{Department of Mathematics and Computer Science, College of the Holy Cross, Worcester, MA 01610, USA} 
\email{cballant@holycross.edu}

\address{Department of Mathematics and Statistics, Amherst College, Amherst, MA 01002, USA}  
\email{afolsom@amherst.edu}

\begin{document}

\maketitle
\begin{center}\textit{Dedicated to Professor M.V. Subbarao in honor of the centenary of his birth.} \end{center}
\begin{abstract}
The celebrated Rogers-Ramanujan identities equate the number of integer partitions of $n$ ($n\in\mathbb N_0$) with parts congruent to $\pm 1 \pmod{5}$ (respectively $\pm 2 \pmod{5}$) and  the number of partitions of $n$ with super-distinct parts (respectively  super-distinct parts greater than $1$).   In this paper, we establish companion identities to the Rogers-Ramanujan identities on the number of parts in all partitions  of $n$ of the aforementioned types, in the spirit of earlier work by Andrews and Beck on a partition identity of Euler. 
\end{abstract}
\section{Introduction}\label{sec_intro}
The Rogers-Ramanujan identities are a pair of identities which assert that the number of integer partitions of $n$ ($n\in\mathbb N_0$) with parts congruent to $\pm 1 \pmod{5}$ (respectively $\pm 2 \pmod{5}$) equals the number of partitions of $n$ with super-distinct parts (respectively  super-distinct parts {greater than $1$}).  Super-distinct parts are also referred to as $2$-distinct parts, and must    differ by at least 2.    The   identities are typically expressed in analytic form, as
\begin{align*}
\sum_{n=0}^\infty \frac{q^{n^2}}{(q;q)_n} &= \frac{1}{(q;q^5)_\infty (q^4;q^5)_\infty}, \\
\sum_{n=0}^\infty \frac{q^{n^2+n}}{(q;q)_n} &=   \frac{1}{(q^2;q^5)_\infty (q^3;q^5)_\infty}, 
\end{align*}
noting that the series and products appearing are the relevant partition generating functions.  Here and throughout, the $q$-Pochhammer symbol is defined for $n\in \mathbb N_0 \cup\{\infty\}$ by   
\begin{align*}
 (a;q)_n := \prod_{j=0}^{n-1}(1-a q^j) = (1-a)(1-aq)(1-aq^2) \cdots (1-aq^{n-1}).
\end{align*} 
For the remainder of the article, we assume $|q|<1$ so that all series converge absolutely.

The Rogers-Ramanujan identities have an extensive and rich history.  Rogers and Ramanujan independently discovered the identities in the late 19th century/early 20th century,  and Rogers provided the first known proof \cite{Rogers}.    Rogers and Ramanujan later published a joint proof \cite{RogRam}, around the same time that Schur independently rediscovered and proved the identities \cite{Schur}.     Since then, the identities have played important roles in and have made connections to diverse areas, including combinatorics, $q$-hypergeometric series, Lie Algebras, modular forms, statistical mechanics, and more (see, e.g., \cite{AndRRG, AndHex, Duke, Fol, GOW, LepWil2, LepWil1, Sills, Slater}, for more).

Like the Rogers-Ramanujan identities, many other  identities in the subject of integer partitions equate the number of partitions of $n$ with parts belonging to a certain set and the number of partitions of $n$ satisfying a particular condition.   Perhaps the oldest such result is Euler's identity, which equates  the number of partitions of $n$ with odd parts and the number of partitions of $n$ with distinct parts.  {Centuries later in 2017,}  Beck made the following related conjecture concerning the number of parts in all partitions of the types appearing in Euler's identity, which we state as follows \cite{oeisA090867}, \cite[Conjecture]{A17}.  
\begin{conj}[Beck] \label{conj}
The excess of the number of parts in all partitions of $n$ with odd parts over the number of parts in all partitions of $n$ with distinct parts equals the number of partitions of $n$ with exactly one even part (possibly repeated).
\end{conj}  
Andrews \cite{A17} quickly proved Beck's conjecture, and additionally showed that this excess also equals the number of partitions of $n$ with exactly one  part  repeated (and all other parts distinct).   Yang \cite{Yang19} and Ballantine--Bielak \cite{BB19} also provided independent combinatorial proofs of Beck's conjecture.   This  work on Beck's conjecture on the number of parts in all partitions of the types appearing in Euler's identity has been followed by a number of generalizations  and Beck-type  companion identities to other well known identities, such as   \cite{AB19, BW21, LW20, Yang19}.   In \cite{BW21a}, Beck-type identities are generalized to all Euler  pairs of order $r$ as defined by Subbarao in \cite{S71}.

In this paper, we state and prove Beck-type companion identities to the Rogers-Ramanujan identities on the excess 
of the number of parts in all partitions of $n$ with parts congruent to $\pm 1 \pmod{5}$ (respectively $\pm 2 \pmod{5}$) over the number of parts in all partitions of $n$ with super-distinct parts (respectively super-distinct parts greater than $1$).  These results are stated in Theorem \ref{thm_rr1} and Theorem \ref{thm_rr2} below, and we give proofs {which are both analytic and combinatorial in nature} in the sections that follow.   

\section{Preliminaries} \label{sec:notation} 
{In this section, we give some background and preliminaries on partitions and $q$-series.}
\subsection{Integer partitions}
Let $n\in\mathbb N_0$. A \emph{partition} of $n$, denoted  $\lambda=(\lambda_1, \lambda_2, \ldots, \lambda_j)$,     is a non-increasing sequence of positive integers $\lambda_1\geq \lambda_2 \geq \cdots \geq \lambda_j$ called \emph{parts} that add up to $n$. We refer to $n$ as the \emph{size} of $\lambda$.  The length of $\lambda$ is the number of parts of $\lambda$, denoted by $\ell(\lambda)$.   We abuse notation and use $\lambda$ to denote either the multiset of its parts or the non-increasing sequence of parts.  We write $a\in \lambda$ to mean the positive integer $a$ is a part of $\lambda$. We write $|\lambda|$ for the size of $\lambda$ and  $\lambda\vdash n$  to mean that $\lambda$ is a partition of size $n$. For a pair of partitions $(\lambda, \mu)$ we also write $(\lambda, \mu)\vdash n$ to mean $|\lambda|+|\mu|=n$. 
We use the convention that $\lambda_k=0$ for all $k>\ell(\lambda)$.  When convenient we will also use the exponential notation for parts in a partition:  the exponent of a part is the multiplicity of the part in the partition.  This notation will be used mostly for rectangular partitions.  We write $(a^b)$ for the partition consisting of $b$ parts equal to $a$. 

The \emph{Ferrers diagram} of a partition $\lambda=(\lambda_1, \lambda_2, \ldots, \lambda_j)$ is an array of left justified boxes such that the $i$th row from the top contains $\lambda_i$ boxes. We abuse notation and use $\lambda$ to mean a partition or its Ferrers diagram.  For example, the Ferrers diagram of $\lambda=(5, 2,2,1)$ is shown in Figure \ref{fig:1}. 
\begin{figure}
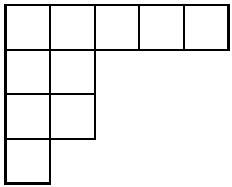

$$\small\ydiagram{5, 2,2,1}$$ 
\caption{A Ferrers diagram}
 \label{fig:1}
\end{figure}

Given a partition $\lambda$, its \emph{conjugate} $\lambda'$ is the partition for which the rows in its Ferrers diagram are precisely the columns in the Ferrers diagram of $\lambda$. For example, the conjugate of $\lambda=(5, 2,2,1)$ is $\lambda'=(4,3,1,1,1)$. 

By the sum of the partitions $\lambda=(\lambda_1, \lambda_2, \ldots, \lambda_j)$ and $\mu=(\mu_1, \mu_2, \ldots, \mu_k)$ we mean the partitions $\lambda+\mu=(\lambda_1+\mu_1, \lambda_2+\mu_2, \ldots, \lambda_\ell+\mu_\ell)$, where $\ell=\max\{j,k\}$. 

{As mentioned in Section \ref{sec_intro},} we say that the parts of a partition are \textit{super-distinct} if any two parts differ by at least $2$. We refer to partitions with super-distinct parts as super-distinct partitions. 

Since our goal is to study the number of parts in partitions, we introduce the notion of  marked partitions. A \textit{marked partition} is a partition with a single part marked. Note that $(5,2^*, 2, 1)$ and $(5,2,2^*,1)$ are different marked partitions.  Then the number of parts in all partitions of $n$ satisfying certain conditions is equal to the number of marked partitions of $n$ satisfying the same conditions. 

If $\mathcal  M(n)$ is the set of marked partitions of size $n$ whose parts satisfy certain conditions, we have a one-to-one correspondence between $\mathcal  M(n)$ and the set of pairs of partitions $(\lambda, (a^b))\vdash n$, where $b$ is a positive integer and $\lambda$ and $(a)$ are partitions whose parts satisfy the conditions of $\mathcal M(n)$. {To explain this, if} $\mu\in \mathcal M(n)$  has marked part $\mu_j=a$ and $\mu_j$ is the $b$th part equal to $a$,  we remove from $\mu$ the first $b$ parts equal to $a$  to obtain a partition $\lambda$. Then $\mu \leftrightarrow (\lambda, (a^b))$. Thus,  the number of parts in all partitions of $n$ satisfying certain conditions is equal to the number of pairs of partitions $(\lambda, (a^b))\vdash n$, $a,b>0$, such that $a$ and the parts of $\lambda$ satisfy the same conditions. 

For more details on partitions, we refer the reader to \cite{AndrewsEncy,AE}.
\subsection{Some results on $q$-series}
The $q$-binomial coefficients $\left[\begin{array}{c} A+k \\ k \end{array} \right]_q$  may be defined as the generating function for the number of partition of $n$ with at most $k$ parts, each part at most $A$ \cite[p67]{AE}, from which it follows that
\begin{align}\label{lem_1}  
 \sum_{0\leq n_1\leq n_2 \cdots    \leq n_k\leq A} q^{n_1+n_2 + \cdots + n_k}
= \left[\begin{array}{c} A+k \\ k \end{array} \right]_q,
\end{align}
\begin{align}\label{qbin_sym}
\left[\begin{array}{c} A+k \\ k \end{array} \right]_q = \left[\begin{array}{c} A+k \\ A \end{array} \right]_q,
\end{align} and
\begin{align}\label{qbin_zero}
\left[\begin{array}{c} A+k \\ k \end{array} \right]_q = 0, \text{ if $A<0$, $k<0$, or $A+k<0$}. 
\end{align}
The $q$-binomial series \cite[Theorem 9]{AE} gives the following generating function for the $q$-binomial coefficients ($|t|<1$)
\begin{align}\label{lem_3}\sum_{k=0}^\infty \left[\begin{array}{c} A+k \\ k \end{array} \right]_q t^k = \frac{1}{(t;q)_{A+1}}.
\end{align}
Another $q$-series identity we will make use of is 
\begin{align}\label{lem_2} 
 \sum_{A\leq n_1\leq n_2 \cdots    \leq n_k} q^{n_1+n_2 + \cdots + n_k} = 
\frac{q^{Ak}}{(q;q)_k},
\end{align}  which can be verified  directly analytically, or combinatorially 
by viewing a partition $\lambda$ with $k$ parts greater than or equal to $A>0$ as $\lambda=\eta+(A^k)$ where $\eta$ is a partition with at most $k$ parts. If $A=0$, \eqref{lem_2} is the generating function for partitions with at most $k$ parts. 
\section{The number of parts in the first Rogers-Ramanujan identity}

Our first result, Theorem \ref{thm_rr1}, gives the excess in the number of parts of partitions involved in the first Rogers-Ramanujan identity. We consider the empty partition a super-distinct partition.

\begin{theorem}\label{thm_rr1} The excess 
of the number of parts in all partitions of $n$ with parts congruent to $\pm 1 \pmod{5}$  over the number of parts in all super-distinct partitions of $n$  equals the number of pairs of partitions $(\lambda, (a^b))$ satisfying all of the following conditions: 
$\lambda$ is a super-distinct partition of $n-ab$,  $a\equiv \pm 1\pmod 5$, $b\geq 1$, and 
if $a=1$, then at least one of $b-1, b, b+1$ is a part of $\lambda$. 
\end{theorem}

 Before we prove the theorem we note that  the original Beck identity {(Conjecture \ref{conj})} can be reformulated in terms of pairs of partitions {as in Theorem \ref{thm_rr1} above} as follows. The excess in the total number of parts in all partitions of $n$ into distinct parts over the  total number of parts in all partitions of $n$ into odd parts equals the number of pairs of partitions $(\lambda, (a^b))\vdash n$, where $\lambda$ is a partition into odd parts, $a,b>0$ and $a$ is even.  This is also the number of pairs $(\lambda, (a^b))\vdash n$, where $\lambda$ is a partition into distinct parts, $a>0$, $a\not\in \lambda$, and $b\geq 2$. 
 
\begin{example} Let $n=4$. The partitions with parts congruent to $\pm 1 \pmod 5$ are $(4)$ and $(1,1,1,1)$ and thus there are five parts in these partitions. The partitions into super-distinct parts are $(4)$ and $(3,1)$ and there are three parts in these partitions. There are two pairs of partitions satisfying the conditions of the theorem: $((2), (1^2))$ and $(\emptyset, (4))$.  \end{example}

We provide two proofs of  Theorem \ref{thm_rr1} below, the first of which is analytic, and the second of which is combinatorial.  \subsection{Analytic proof of Theorem \ref{thm_rr1}} 
Let $a(n,m)$ denote the number of partitions of $n$ with parts congruent to $\pm 1 \pmod{5}$ and exactly $m$ parts. Then, 
the generating function for $a(n,m)$ is given by
$$P_1(z;q):= \sum_{n\geq 0}\sum_{m\geq 0} a(n,m)z^mq^n= \frac{1}{(zq;q^5)_\infty(z q^4;q^5)_\infty}.$$
Similarly, if $b(n,m)$ is the number partitions of $n$ with super-distinct parts and exactly $m$ parts, the generating function for $b(n,m)$ is  
$$R_1(z;q):=  \sum_{n\geq 0}\sum_{m\geq 0} b(n,m)z^mq^n=\sum_{n=0}^\infty \frac{z^n q^{n^2}}{(q;q)_n}.$$
Considering  the difference of the derivatives of these functions with respect to $z$  evaluated at $1$, we obtain the generating function for the excess in the number of parts {in all partitions of $n$ with parts congruent to $\pm 1 \pmod{5}$  over the number of parts in all super-distinct partitions of $n$.} We have 
\begin{align}\notag \frac{\partial}{\partial z} & \Big|_{z=1}(P_1(z;q)  -R_1(z;q) ) \\ \notag & = \frac{1}{(q;q^5)_\infty(q^4;q^5)_\infty} \left(\sum_{m=1}^\infty \frac{q^{5m-1}}{1- q^{5m-1}} + \frac{q^{5m-4}}{1- q^{5m-4}} \right)- \sum_{n=0}^\infty \frac{n q^{n^2}}{(q;q)_n} \\ \notag & =  \left(\sum_{n=0}^\infty \frac{q^{n^2}}{(q;q)_n}\right)\left(\sum_{m=1}^\infty \frac{q^{5m-1}}{1- q^{5m-1}} + \frac{q^{5m-4}}{1- q^{5m-4}} \right)- \sum_{n=0}^\infty \frac{n q^{n^2}}{(q;q)_n}\\ \label{def_T1T2} & =: T_1(q)-T_2(q).\end{align}

We next write  down five different generating functions such that their sum is the generating function for the number of pairs of partitions $(\lambda, (a^b))$ with $|\lambda| + ab = n$ satisfying all of the conditions given in Theorem \ref{thm_rr1}).  After doing so, we will prove that the resulting sum of generating functions is equal to $T_1(q) - T_2(q)$. \smallskip
\ \\  
{\bf{Case 1.}} The generating function for the number of pairs of partitions $(\lambda, (1^b))$  such that $\lambda$ is a  super-distinct partition of $n-b$, and $b-1 \in \lambda, b+1 \not \in \lambda$ is
\begin{align}\label{c1_gf} 
\sum_{b=1}^\infty q^b \sum_{m=1}^\infty \sum_{j=1}^m \mathop{\sum_{0\leq n_1\leq n_2 \cdots   \leq n_{j-1}\leq  n_j <  n_{j+1} \leq n_{j+2}   \cdots \leq n_m}}_{2j-1 + n_j = b-1} q^{(1+n_1)+(3+n_2) + \cdots + (2m-1 + n_m)}
,\end{align} which we explain as follows.  Any super-distinct partition with $m\geq1$ parts is of the form $(2m-1+n_m, 2m-3 + n_{m-1},\dots,3 + n_2, 1+n_1)$, where $0\leq n_1\leq n_2 \leq \cdots \leq n_m$.  We sum over all possible positions for a specified part $b-1$ in such a position, namely $2j-1+n_j = b-1$ for $1\leq j \leq m$.  Since $b+1$ can not be in such a partition, we must have the the difference between the consecutive parts
$2j+1 + n_{j+1}$ and $2j-1 + n_j=b-1$ is at least $3$;   equivalently, $n_j < n_{j+1}$.     The size of the second partition $1^b$ in a pair $(\lambda,(1^b))$  appears in the exponent of $q^b$, and we sum over all possible $b\geq 1$.  

We re-write the inner sum in \eqref{c1_gf} as
\begin{align*}
q^{m^2 + b-2j} &\left(\sum_{0\leq n_1\leq n_2\leq  \cdots   \leq n_{j-1}\leq  b-2j}  q^{n_1+n_2 + \cdots + n_{j-1}}\right)\\ &\hspace{.1in}\times \left(\sum_{b-2j+1\leq n_{j+1}\leq n_{j+2}\leq \cdots   \leq n_{m}}  q^{n_{j+1}+n_{j+2} + \cdots + n_m}\right),
\end{align*} where we have also used that $1+3+5+\cdots + 2m-1 = m^2$.   
Using \eqref{lem_1} and \eqref{lem_2}, this can be rewritten as
\begin{align}\label{eqn_inn}
q^{m^2 + b-2j} \left[ \begin{array}{c} b-j-1 \\ j-1 \end{array} \right]_q \frac{q^{(b-2j+1)(m-j)}}{(q;q)_{m-j}}.
\end{align}  Using \eqref{eqn_inn}, the generating function in \eqref{c1_gf} becomes
\begin{align}\notag
\sum_{b=1}^\infty & q^b \sum_{m=1}^\infty \sum_{j=1}^m 
q^{m^2 + b-2j} \left[ \begin{array}{c} b-j-1 \\ j-1 \end{array} \right]_q \frac{q^{(b-2j+1)(m-j)}}{(q;q)_{m-j}} \\ \notag&=  \sum_{m=1}^\infty \sum_{j=1}^m 
\frac{q^{m^2 -2j + (-2j+1)(m-j)} }{(q;q)_{m-j}} \sum_{b=1}^\infty \left[ \begin{array}{c} b-j-1 \\ b- 2j \end{array} \right]_q  q^{b(m-j+2) } \\ \notag
&=  \sum_{m=1}^\infty \sum_{j=1}^m 
\frac{q^{m^2 +m+j}}{(q;q)_{m-j}} \sum_{b=0}^\infty \left[ \begin{array}{c} b+j-1 \\ b \end{array} \right]_q q^{b(m-j+2)}  \\  \notag & = 
 \sum_{m=1}^\infty \sum_{j=1}^m 
\frac{q^{m^2 +m+j }}{(q;q)_{m-j}(q^{m-j+2};q)_j} \\ \label{c1_gf2}
&= \sum_{m=1}^\infty \frac{q^{m^2}}{(q;q)_{m+1}}\sum_{j=1}^m 
 q^{m+j} (1-q^{m-j+1}),  
\end{align}  where we have also used   \eqref{qbin_sym}, \eqref{qbin_zero} and  \eqref{lem_3}.  
 \medskip \ \\
{\bf{Case 2.}} By an explanation similar to the one given in Case 1, the generating function for the number of pairs of partitions $(\lambda, (1^b))$  such that $\lambda$ is a super-distinct partition of $n-b$, and $b+1 \in \lambda, b-1 \not \in \lambda$ is
\begin{align}\label{c2_gf} 
\sum_{b=1}^\infty q^b \sum_{m=1}^\infty \sum_{j=1}^m \mathop{\sum_{0\leq n_1\leq n_2\leq  \cdots   \leq n_{j-1}<  n_j \leq  n_{j+1}\leq    \cdots \leq n_m}}_{2j-1 + n_j = b+1} q^{(1+n_1)+(3+n_2) + \cdots + (2m-1 + n_m)}.\end{align}
Arguing as in Case 1 and using   \eqref{lem_1}--\eqref{lem_2}, we obtain that this equals
\begin{align}\label{c2_gf2} \sum_{m=1}^\infty \frac{q^{m^2}}{(q;q)_{m+1}}\sum_{j=1}^m 
 q^{m+j} (1-q^{m-j+1}).\end{align} \ \\
{\bf{Case 3.}} Similar to the previous cases, we have that the generating function for the number of pairs of partitions $(\lambda, (1^b))$  such that $\lambda$ is a super-distinct partition of $n-b$, and $b-1, b+1 \in \lambda$ is
\begin{align}\label{c3_gf} 
\sum_{b=1}^\infty q^b \sum_{m=1}^\infty \sum_{j=2}^m \mathop{\mathop{\sum_{0\leq n_1\leq n_2 \cdots   \leq n_{j-1}\leq  n_j \leq  n_{j+1} \leq \cdots \leq n_m}}_{2j-3 + n_{j-1} = b-1}}_{2j-1 + n_{j} = b+1}q^{(1+n_1)+(3+n_2) + \cdots + (2m-1 + n_m)}.\end{align}
Arguing as in Case 1 and  using  \eqref{lem_1}--\eqref{lem_2}, we obtain that this equals
\begin{align}\label{c3_gf2} \sum_{m=1}^\infty \frac{q^{m^2}}{(q;q)_{m+1}}\sum_{j=2}^m 
 q^{2j-2} (1-q^{m-j+1})(1-q^{m-j+2}).\end{align} \ \\
{\bf{Case 4.}}  Similar to the previous cases, we have that the generating function for the number of pairs of partitions $(\lambda, (1^b))$  such that $\lambda$ is a super-distinct partition of $n-b$, and $b \in \lambda$ is
\begin{align}\label{c4_gf} 
\sum_{b=1}^\infty q^b \sum_{m=1}^\infty \sum_{j=1}^m \mathop{\sum_{0\leq n_1\leq n_2 \cdots   \leq n_{j-1}\leq  n_j \leq  n_{j+1} \leq n_{j+2}   \cdots \leq n_m}}_{2j-1 + n_j = b} q^{(1+n_1)+(3+n_2) + \cdots + (2m-1 + n_m)}.\end{align}
Arguing as in Case 1 using \eqref{lem_1}--\eqref{lem_2}, we obtain that this equals
\begin{align}\label{c4_gf2} \sum_{m=1}^\infty \frac{q^{m^2}}{(q;q)_{m+1}}\sum_{j=1}^m 
 q^{2j-1} (1-q^{m-j+1}).\end{align} 
{\bf{Case 5.}}
It is not difficult to see that generating function for the number of pairs of partitions $(\lambda, (a^b))$ with   $a\equiv \pm 1 \pmod{5}$ and $a>1$ such that $\lambda$ is a super-distinct partition of $n-ab$  is
\begin{align}
\label{c5_gf}  \left(\sum_{m=0}^\infty \frac{q^{m^2}}{(q;q)_m}\right)\left(\frac{q^4}{1-q^4} + \sum_{m=2}^\infty \left(\frac{q^{5m-1}}{1-q^{5m-1}} + \frac{q^{5m-4}}{1-q^{5m-4}} \right)\right), 
\end{align} using that $R_1(1;q)$ is the generating function for super-distinct partitions. \smallskip

Armed with the generating functions in Cases 1--5 above, we now complete the analytic proof of Theorem \ref{thm_rr1}.  The pairs of partitions described in Theorem \ref{thm_rr1} may be realized as a disjoint union of the pairs described in Cases 1--5 above.  Thus, the generating function for the excess described in Theorem \ref{thm_rr1} may be realized as the sum of the generating functions given in \eqref{c1_gf2}, \eqref{c2_gf2}, \eqref{c3_gf2}, \eqref{c4_gf2}, and \eqref{c5_gf}.  We first add \eqref{c1_gf2}, \eqref{c2_gf2}, \eqref{c3_gf2}, and \eqref{c4_gf2} to obtain:
 \begin{multline*} 2\sum_{m=1}^\infty \frac{q^{m^2}}{(q;q)_{m+1}}\sum_{j=1}^m 
 q^{m+j} (1-q^{m-j+1})  \\ + \sum_{m=1}^\infty \frac{q^{m^2}}{(q;q)_{m+1}}\sum_{j=1}^{m-1} 
 q^{2j} (1-q^{m-j})(1-q^{m-j+1})  \\ + 
\sum_{m=1}^\infty \frac{q^{m^2}}{(q;q)_{m+1}}\sum_{j=1}^m 
 q^{2j-1} (1-q^{m-j+1}) \end{multline*}
 \begin{align}\notag &=\sum_{m=1}^\infty    \frac{q^{m^2}}{(q;q)_{m+1}}\\ \notag 
 &  \times \Bigg(\sum_{j=1}^{m-1}  \bigg(2q^{m+j} (1-q^{m-j+1})  +  q^{2j} (1-q^{m-j})(1-q^{m-j+1}) +  q^{2j-1} (1-q^{m-j+1})\bigg)  \\  \notag  & {\hspace{1.6in}} 
 + 2 q^{2m}(1-q) + q^{2m-1}(1-q)\Bigg) \\\notag  
 & = \sum_{m=1}^\infty \frac{q^{m^2}}{(q;q)_{m+1}}  \Bigg(\sum_{j=1}^{m-1}  \bigg(q^{2j}+q^{2j-1} - q^{m+j+1} - q^{2m+1}\bigg) \\ \notag & \hspace{1.6in}
 + 2 q^{2m}(1-q) + q^{2m-1}(1-q)\Bigg)  \\ \notag 
 &= \sum_{m=1}^\infty \frac{q^{m^2}}{(q;q)_{m+1}}  \Bigg(
 q \frac{1-q^{2m-2}}{1-q} - q^{m+2}\frac{1-q^{m-1}}{1-q} \\ \notag &\hspace{1.6in} - (m-1)q^{2m+1}
  + 2 q^{2m}(1-q) + q^{2m-1}(1-q)\Bigg) \\ \notag  &=
  \sum_{m=1}^\infty \frac{q^{m^2}}{(q;q)_{m+1}}  \Bigg(
 q \frac{1-q^{m+1}}{1-q}   - (m+1)q^{2m+1}\Bigg) \\ \notag
 &=
  \frac{q}{1-q}\sum_{m=1}^\infty \frac{q^{m^2}}{(q;q)_{m}}    - \sum_{m=2}^\infty \frac{m q^{m^2}}{(q;q)_m}
  \\  \label{exc_1234} &=\frac{q}{1-q}\sum_{m=0}^\infty \frac{q^{m^2}}{(q;q)_{m}}    - T_2(q).
  \end{align} Adding \eqref{c5_gf}  to $\frac{q}{1-q}\sum_{m=0}^\infty \frac{q^{m^2}}{(q;q)_{m}}$ from 
  \eqref{exc_1234} 
 we obtain $T_1(q)$, which completes the proof.    \qed
\begin{remark} One can also see  that \eqref{c1_gf2}, \eqref{c2_gf2}, \eqref{c3_gf2}, and \eqref{c4_gf2} are the  respective generating functions {for the pairs of partitions as at the start of Cases 1-4 above} by viewing partitions into $m$ super-distinct parts as the sum of an odd staircase of length $m$, $\delta_m=(2m-1, 2m-3, \ldots, 3,1)$, and the conjugate of a partition with parts at most $m$. We explain this for \eqref{c4_gf2}, noting that \eqref{c1_gf2}, \eqref{c2_gf2}, and \eqref{c3_gf2} can be interpreted similarly.

To show {combinatorially} that \eqref{c4_gf2} is the generating function for the number of pairs of partitions $(\lambda, (1^b))\vdash n$  such that $\lambda$ has super-distinct parts, and $b \in \lambda$, we observe that in 
\begin{align*} \sum_{m=1}^\infty \sum_{j=1}^m q^{2j-1}\cdot q^{m^2} \cdot \frac{1-q^{m-j+1}}{(q;q)_{m+1}},\end{align*} 
for fixed $m,j$, 
the term $\displaystyle \frac{1-q^{m-j+1}}{(q;q)_{m+1}}$ generates partitions with parts at most $m+1$ and no part equal to $m-j+1$. By conjugation, it generates partitions $\eta$ with at most $m+1$ parts and  $\eta_{m-j+1}=\eta_{m-j+2}$. 
The term $q^{m^2}$ generates the staircase $\delta_m$ and the term $q^{2j-1}$ generates the partition $(1^{2j-1})$. Thus, $\displaystyle q^{2j-1}\cdot q^{m^2} \cdot \frac{1-q^{m-j+1}}{(q;q)_{m+1}}$ generates triples $(\delta_m, \eta, (1^{2j-1}))$. Each such  triple  corresponds to the pair $(\lambda, (1^b))$, where $b=2j-1+\eta_{m-j+2}$ and $\lambda=\delta_m+(\eta\setminus\{\eta_{m-j+2}\})$ has $m$ super-distinct parts and $\lambda_{m-j+1}=2j-1+\eta_{m-j+1}=b$.
\end{remark}
\subsection{Combinatorial proof of Theorem \ref{thm_rr1}}

We interpret $T_1(q)$ defined in \eqref{def_T1T2} as the generating function for $|\mathcal T_1(n)|$, where $\mathcal T_1(n)$ is the set of pairs of partitions 
$(\lambda, (a^b))\vdash n$ such that   
$\lambda$ is a super-distinct partition,   $a\equiv \pm 1\pmod 5$, and $b\geq 1$.  By the first Rogers-Ramanujan identity, $|\mathcal T_1(n)|$ is also the number of pairs of partitions $(\lambda, (a^b))\vdash n$ such that   
$\lambda$ is a partition  into parts congruent to $\pm 1\pmod 5$, $a\equiv \pm 1\pmod 5$, and $b\geq 1$. Then, as explained in Section \ref{sec_intro}, $|\mathcal T_1(n)|$ equals the number of parts in all partitions of $n$ into parts congruent to $\pm 1\pmod 5$.

We interpret $T_2(q)$ defined in \eqref{def_T1T2} as the generating function for $|\mathcal T_2(n)|$, where $\mathcal T_2(n)$ is the set of marked partitions of $n$ with super-distinct parts.   Thus, as explained in Section \ref{sec_intro}, $|\mathcal T_2(n)|$ equals the number of parts in all partitions of $n$ into super-distinct parts.

To prove Theorem \ref{thm_rr1} combinatorially, we create an injection  $\varphi: \mathcal T_2(n)\to  \mathcal T_1(n)$ as follows. If $\mu\in  \mathcal T_2(n)$ has marked part $\mu_i$, then $$\varphi(\mu):= (\mu\setminus \{\mu_i\}, (1^{\mu_i})).$$ {In terms of Ferrers diagrams, $\varphi$ removes the row of length $\mu_i$ and transforms it into a rectangular partition $(a^b)$ with $a=1$ and $b=\mu_i$.} 
The image $\varphi(\mathcal T_2(q))$ of the injection consists of pairs $(\lambda, (1^b))\in \mathcal T_1(n)$ such that  none of  $b-1, b, b+1$ is  a part of $\lambda$.  {The inverse of $\varphi$ on $\varphi(\mathcal T_2(q))$, takes $(\lambda, (1^b))\in \mathcal T_1(n)$ such that $b-1, b, b+1\not\in \lambda$, and creates a marked partition $\mu$ by inserting a part equal to $b$ into $\lambda$ and marking  it. }

Then the excess  
of the number of parts in all partitions of $n$ with parts congruent to $\pm 1 \pmod{5}$  over the number of parts in all partitions of $n$ with super-distinct parts equals the size of $\mathcal T_2(n)\setminus \varphi(\mathcal T_2(n))$, the set of pairs of partitions $(\lambda, (a^b))$ such that
$\lambda$ is a super-distinct partition of $n-ab$,  $a\equiv \pm 1\pmod 5$, $b\geq 1$, and 
if $a=1$, then at least one of $b-1, b, b+1$ is a part of $\lambda$. 
\qed

\begin{corollary} \label{cor_rr}Let $n\geq 1$. The number of parts in all partitions of $n$ into super-distinct parts is less than the number of parts equal to $1$ in all partitions of $n$ into parts congruent to $\pm 1 \pmod 5$. 
\end{corollary}

\begin{proof} From the argument in the introduction, the number of parts equal to $1$ in all partitions of $n$ into parts congruent to $\pm 1 \pmod 5$ is equal to $|{\mathcal T}_{1,1}(n)|$, where $\mathcal T_{1,1}(n)$ is the set of marked partitions of $n$ into  parts congruent to $\pm 1 \pmod 5$ with at least one part equal to $1$ and in which the marked part is one of the parts equal to $1$. 


Let $\varphi$ be the injection defined  in the proof of Theorem \ref{thm_rr1}.  If $\mu \in  \mathcal T_2(n)$  with $\mu_j$ marked, then    $\varphi(\mu)=( \mu\setminus \{\mu_j\}, (1^{\mu_j}))$. Since $ \mu\setminus \{\mu_j\}$ is a partition of $n-\mu_j$ into super-distinct parts, by the first Rogers-Ramanujan identity, it corresponds to a unique partition $\eta$ of  $n-\mu_j$ into parts congruent to $\pm 1 \pmod 5$. Consider the marked partition $\xi(\mu) = \eta \cup (1^{\mu_j})$ with the $\mu_j$th part  equal to $1$ marked. This gives an injection from $\mathcal T_2(n)$ into $\mathcal T_{1,1}(n)$. \end{proof}

\begin{remark}  The injection in the combinatorial proof of Theorem \ref{thm_rr1} above establishes that the number of marked super-distinct partitions of $n$ 
equals the number of pairs of partitions $(\lambda,(1^b)) \vdash n$ such that $\lambda$ is a partition  into super-distinct parts, $b\geq 1$, and none of $b-1, b, b+1$ is in $\lambda$ (equivalently, the difference between the number of pairs of partitions $(\lambda,(1^b)) \vdash n$ such that $\lambda$ is a partition  into super-distinct parts, $b\geq 1$,  and the number of such pairs with at least one of of $b-1, b, b+1$ in $\lambda$).  
This same identity also follows independently from the analytic proof of Theorem \ref{thm_rr1}, which we explain as follows.  We have that \eqref{exc_1234}  is the generating function for the difference between the number of pairs of partitions $(\lambda,(1^b)) \vdash n$ such that $\lambda$ is a partition  into super-distinct parts, $b\geq 1$, and the number of marked super-distinct partitions of $n$ 
On the other hand,    \eqref{exc_1234} originated as the sum of 
\eqref{c1_gf}, \eqref{c2_gf}, \eqref{c3_gf}, and \eqref{c4_gf}, a sum which is the generating function for the number of pairs of partitions $(\lambda,(1^b)) \vdash n$ such that $\lambda$ is a partition  into super-distinct parts, $b\geq 1$, and at least one of $b-1, b, b+1$ is in $\lambda$.   Equating these two interpretations for the $q$-series coefficients of \eqref{exc_1234} gives the (equivalent) identity due to the injection in the combinatorial proof of Theorem \ref{thm_rr1}.
\end{remark}
\section{The number of parts in the second Rogers-Ramanujan identity}
In this section we formulate  and prove a result for the second Rogers-Ramanujan identity that is analogous to Theorem \ref{thm_rr1}. Somewhat surprisingly, it is  more difficult to establish this theorem.  As our proof of Theorem \ref{thm_rr2} will show, the excess 
of the number of parts in all partitions of $n$ with   parts congruent to $\pm 2 \pmod{5}$ over the number of parts in all partitions of $n$ with super-distinct parts greater than $1$ can be described combinatorially as the size of a subset $\mathcal S(n)$  (see \eqref{def_Snset} below) of the set of pairs or partitions $(\lambda, (a^b))\vdash n$ such that $\lambda$ has super-distinct parts  {greater than $1$},  $a\equiv \pm 2\pmod 5$, $b\geq 1$. The conditions satisfied by the pairs of partition in $\mathcal S(n)$ can be stated explicitly.  This rather long list of conditions is built around  residue classes of $a$ and the interplay between $a, b,$ and certain parts of $\lambda$, and we do not present it here in its explicit form {for brevity's sake.}  


\begin{theorem}\label{thm_rr2} The excess 
of the number of parts in all partitions of $n$ with   parts congruent to $\pm 2 \pmod{5}$ over the number of parts in all partitions of $n$ with super-distinct parts  {greater than $1$}   equals the number of pairs or partitions $(\lambda, (a^b))\vdash n$ such that $\lambda$ has super-distinct parts  {greater than $1$},  $a\equiv \pm 2\pmod 5$, $b\geq 1$, {{and satisfying conditions prescribed by $\mathcal S(n)$.}}  \end{theorem}

\begin{proof} Let $c(n,m)$ denote the number of partitions of $n$ with parts congruent to $\pm 2 \pmod{5}$ and exactly $m$ parts. Then, 
the generating function for $c(n,m)$ is given by
$$P_2(z;q):= \sum_{n\geq 0}\sum_{m\geq 0} c(n,m)z^mq^n= \frac{1}{(zq^2;q^5)_\infty(z q^3;q^5)_\infty}.$$
Similarly, if $d(n,m)$ is the number partitions of $n$ with super-distinct parts  {greater than $1$}, and exactly $m$ parts 
$$R_2(z;q):=  \sum_{n\geq 0}\sum_{m\geq 0} d(n,m)z^mq^n=\sum_{n=0}^\infty \frac{z^n q^{n^2+n}}{(q;q)_n}.$$

Considering  the difference of the derivatives with respect to $z$  evaluated at $1$, we obtain the generating function for the excess in the number of parts. We have 
\begin{align*}\frac{\partial}{\partial z} & \Big|_{z=1}(P_2(z;q)  -R_2(z;q) ) \\ & = \frac{1}{(q^2;q^5)_\infty(q^3;q^5)_\infty} \left(\sum_{m=1}^\infty \frac{q^{5m-2}}{1- q^{5m-2}} + \frac{q^{5m-3}}{1- q^{5m-3}} \right)- \sum_{n=0}^\infty \frac{n q^{n^2+n}}{(q;q)_n} \\ & =  \left(\sum_{n=0}^\infty \frac{q^{n^2+n}}{(q;q)_n}\right)\left(\sum_{m=1}^\infty \frac{q^{5m-2}}{1- q^{5m-2}} + \frac{q^{5m-3}}{1- q^{5m-3}} \right)- \sum_{n=0}^\infty \frac{n q^{n^2+n}}{(q;q)_n}\\ & =: S_1(q)-S_2(q).\end{align*}
We interpret $S_1(q)$ as the generating function for $|\mathcal S_1(n)|$, where
$\mathcal S_1(n)$ is the set of pairs of partitions  $(\lambda, (a^b)) \vdash n$ such that  $\lambda$ has super-distinct parts greater than $1$,  $a \equiv\pm2 \pmod{5}\}$ and $b\geq 1$ if $n\geq 1$, and $|\mathcal S_1(0)| :=0$. Note that  $\lambda = \emptyset$ is allowed.   As explained in Section \ref{sec_intro}, $|\mathcal S_1(n)|$ is the number of parts in all partitions of $n$ with   parts congruent to $\pm 2 \pmod{5}$.  We interpret $S_2(q)$ as the generating function for $|\mathcal S_2(n)|$, where  $\mathcal S_2(n)$ is the set of marked partitions of $n$ with superdistinct parts greater than $1$  if  $n\geq 1$, and  $|\mathcal S_2(0)|:=0.$   Thus, $|\mathcal S_2(n)|$ is number of parts in all partitions of $n$ with super-distinct parts  {greater than $1$}.

For $n\geq 1$, we create an injection  $\psi: \mathcal S_2(n)\to  \mathcal S_1(n)$ as follows.

Start with  $\mu \in S_2(n)$ and suppose the marked  part of $\mu$ is $\mu_i=c$. Set  $$x:=\begin{cases}\mu_{i+1} & \mbox{ if } i<\ell(\mu)\\ 0 &  \mbox{ if } i=\ell(\mu),\end{cases}$$ and let $y=c-x$. Thus, if the marked part is not the last part of $\mu$, $y$ is the difference between the marked part and the next part. Otherwise, $y$ is equal to the marked part. Hence, $y \geq 2$. Moreover, $\mu$ does not contain any of $c+1, c-1, x+1$ and $x-1$ (if $x\neq 0$) as a part.

We denote by $\widetilde \mu$ the partition obtained from $\mu$ by removing the marked part, i.e, $\widetilde{\mu}:= \mu \setminus \{c\}.$

Our definition of $\psi$ depends on the parity of $c$.

\noindent \underline{Case 1:} $c=2k$, $k\geq 1$. Then, we define $$\psi(\mu):=(\widetilde \mu, (2^{k})).$$ {In terms of Ferrers diagrams, $\psi$  removes the row of length $2k$ from  $\mu$ and transforms it into a rectangular partition $(a^b)$ with  $a=2$ and $b=k$. }

The image under $\psi$ of the subset of overpartitions in $\mathcal S_2(n)$ in this case is    $$\mathcal{I}_1(n):=\{(\lambda, (2^k))\in \mathcal S_1(n) \mid  2k-1, 2k, 2k+1 \not \in \lambda\}.$$
{To see that $\psi$ is onto $\mathcal{I}_1(n)$, given $(\lambda, (2^k))\in \mathcal I_1(n)$ we let $\mu=\lambda\cup \{2k\}$ and mark  part $2k$. Then, $\mu\in \mathcal S_2(n)$ and $\psi(\mu)=(\lambda, (2^k))$.}

\noindent \underline{Case 2:} $c=2k+1$, $k\geq 1$. To define $\psi$ we need to consider different residue classes of $c$ modulo $5$. 

(A)  If $c\equiv 2$ or $3 \pmod 5$, define $$\psi(\mu):=(\widetilde \mu, (c)).$$ {In terms of Ferrers diagrams, $\psi$  removes the row of length $c$ from  $\mu$ and transforms it into a rectangular partition $(a^b)$ with  $a=c$ and $b=1$. }
 
 The image under $\psi$ of the subset of overpartitions in $\mathcal S_2(n)$ in this case is $$\mathcal{I}_2(n):=\{(\lambda, (a))\in \mathcal S_1(n) \mid a \mbox{ odd and } a-1,a,a+1 \not \in  \lambda\}.$$  
 {To see that $\psi$ is onto $\mathcal{I}_2(n)$, given $(\lambda, (a))\in \mathcal I_2(n)$ we let $\mu=\lambda\cup \{a\}$ and mark  part $a$. Then, $\mu\in \mathcal S_2(n)$ and $\psi(\mu)=(\lambda, (a))$.}
 
(B)  If $c\not\equiv 2$ or $3 \pmod 5$, then $c=2k+1 \geq 5$, i.e., $k\geq 2$, and we consider several subcases according to the size of $y$.   

(i) If $y=2$ or $3$, then $x\neq 0$ and we define $$\psi(\mu):=(\widetilde \mu\setminus \{x\}\cup\{x+1\}, (2^k)). $$ {In terms of Ferrers diagrams, $\psi$  removes the row of length $c$ from  $\mu$, adds one to the next  part $x$, and transforms $c-1=2k$  into a rectangular partition $(a^b)$ with  $a=2$ and $b=k$. }

Note that, if $y=2$, then $x+1=2k$, and if $y=3$, then $x+1=2k-1$.

 The image under $\psi$ of the subset of overpartitions in $\mathcal S_2(n)$ in this case is \begin{align*} \mathcal{I}_3(n) := & 
   \left\{(\lambda, (2^k))\in \mathcal S_1(n) \Big |  \begin{array}{l} k\geq 2, 2k \in \lambda \mbox{ and } \\ 2k-2,  2k+2 \not \in  \lambda\end{array}\right \}\bigcup \\ &  \left\{(\lambda, (2^k))\in \mathcal S_1(n) \Big | \begin{array}{l} k\geq 2, 2k-1\in \lambda \mbox{ and } \\ 2k-3,  2k+1, 2k+2 \not \in  \lambda\end{array} \right\}. \end{align*} 
    {Note that in the first set above, we also have that $2k-1$ and $2k+1$ are not parts of $\lambda$. However, this is clear since $2k\in \lambda$ and $(\lambda, (2^k))\in \mathcal S_1(n)$. Similarly, in the second set $2k-2, 2k\not\in\lambda$ but we do not mention this explicitly since it is implied by $2k-1\in \lambda$ and $(\lambda, (2^k))\in \mathcal S_1(n)$. For  the remainder of the proof, we will not write these exclusions explicitly.}
    
   {Clearly, the two sets whose union is $\mathcal I_3(n)$ are disjoint. To see that $\psi$ is onto $\mathcal{I}_3(n)$, let $(\lambda, (2^k))\in \mathcal I_3(n)$. If $2k\in \lambda$, replace part $2k$ by parts $2k-1$ and $2k+1$ and mark $2k+1$.  If $2k\not\in \lambda$, then $2k-1\in \lambda$ and we replace part $2k-1$ by parts $2k-2$ and $2k+1$ and mark  $2k+1$.  We obtain a partition $\mu\in S_2(n)$ such that $\psi(\mu)= (\lambda, (2^k))$.}

(ii)] If $y \geq 4$ and $c\equiv 0\pmod 5$, since $c$ is odd, we write write $c=10j+5=2(5j+2)+1$ with $j\geq 0$. Notice that if $j=0$, then $x=0$. Define $$\psi(\mu):=\begin{cases}(\widetilde \mu\setminus \{x\}\cup\{x+1\}, ((5j+2)^2)) & \mbox{ if } x\neq 0 \\ (\widetilde\mu\cup(5j+3), (5j+2)) & \mbox{ if } x=0.\end{cases} $$
In terms of Ferrers diagrams, if the marked part is not the last part of $\mu$, $\psi$  removes the row of   $\mu$ corresponding to the marked part $c$, adds one to the next  part $x$, and transforms $c-1$  into a rectangular partition $(a^b)$ with  $a=(c-1)/2$ and $b=2$. If the marked part $c$ is the last part of $\mu$, $\psi$ removes the part $c$ from $\mu$, and transforms  it into 
 a new  part equal to  $(c+1)/2$ in $\mu$ and a rectangular partition $(a^b)$ with $a=(c-1)/2$ and $b=1$.
  
{Before we describe the image of $\psi$ in this case, we introduce some helpful notation. For a positive integer $u$, we denote by $z_{u,\lambda}$ the largest part of $\lambda$ that is less than or equal to $u$ and it is implicit in this notation that there is such a part  in $\lambda$.
Then, the image under $\psi$ of the subset of overpartitions in $\mathcal S_2(n)$ in this case is \begin{align*} \mathcal{I}_4(n):= & \left\{(\lambda, (a^2))\in \mathcal S_1(n) \Bigg |\begin{array}{l}  a>2, a\equiv 2\pmod 5 \mbox{ and } \\ 2a-1, 2a, 2a+1, 2a+2 \not \in \lambda\\  z_{2a-2,\lambda} \geq 3 \mbox{ and } z_{2a-2,\lambda}-2  \not \in \lambda\end{array}\right\}\bigcup\\ & \left\{(\lambda, (a))\in \mathcal S_1(n) \Big | \begin{array}{l}  a \equiv 2\pmod 5, \  a+1\in \lambda \mbox{ and } \\ \mbox{if }  z \leq 2a+2, z\neq a+1,  \mbox{ then } z \not \in  \lambda\end{array} \right\}.
   \end{align*}  }
   Clearly, the two sets whose union is $\mathcal I_4(n)$ are disjoint. To see that $\psi$ is onto $\mathcal{I}_4(n)$, let $(\lambda, (a^k))\in \mathcal I_4(n)$. Thus $a\equiv 2\pmod 5$ and $k=1$ or $2$. If $k=2$, we add one to the largest part of $\lambda$ that is less than or equal to $2a-2$ and insert and mark a part equat to $2a-1$ into $\lambda$. If $k=1$, we add $a$ to the smallest part of $\lambda$ and we mark the obtained part. 
  We obtain a partition $\mu\in S_2(n)$ such that $\psi(\mu)= (\lambda, (a^k))$.

(iii) If $y \geq 4$ and $c\equiv 4\pmod 5$, since $c$ is odd, we write write $c=10j+9=2(5j+3)+3$ with $j\geq 0$. Define $$\psi(\mu):=\begin{cases}(\widetilde \mu\setminus \{x\}\cup\{x+3\}, ((5j+3)^2)) & \mbox{ if } x\neq 0 \\ (\widetilde \mu\cup\{3\}, ((5j+3)^2)) & \mbox{ if } x= 0.\end{cases} $$
In terms of Ferrers diagrams,  $\psi$  removes the row of   $\mu$ corresponding to the marked part $c$, adds three to the next  part $x$ if $x\neq 0$ and inserts a parts equal to $3$ into $\mu$ if $x=0$; and transforms $c-3$  into a rectangular partition $(a^b)$ with  $a=(c-3)/2$ and $b=2$.
  
    The image under $\psi$ of the subset of overpartitions in $\mathcal S_2(n)$ in this case is \begin{align*} \mathcal{I}_5&(n)   := \\&   \left\{(\lambda, (a^2))\in \mathcal S_1(n) \Bigg | \begin{array}{l}  a\equiv 3\pmod 5 \mbox{ and } \\  2a+3, 2a+4 \not \in \lambda\\ z_{2a+2, \lambda} \geq 5 \mbox{ and } z_{2a+2, \lambda}-t\not \in \lambda \mbox{ for } t=2,3,4\end{array}\right\}\bigcup \\ & \left\{(\lambda, (a^2))\in \mathcal S_1(n) \Big | \begin{array}{l}  a \equiv 3\pmod 5, \ 3\in \lambda \mbox{ and } \\ \mbox{if }  z \leq 2a+4, z\neq 3,  \mbox{ then } z \not \in  \lambda\end{array} \right\}.
   \end{align*} 
    Considering parts less than or equal to $2a+2$, we see that the two sets whose union is $\mathcal I_5(n)$ are disjoint. To see that $\psi$ is onto $\mathcal{I}_5(n)$, let $(\lambda, (a^2))\in \mathcal I_5(n)$. If $3$ is the only part less than or equal to $2a+4$, remove part $3$ from $\lambda$. Otherwise, subtract $3$ from the largest part of  $\lambda$ that is less than or equal to $2a+2$. Finally insert and mark a part equal to $2a+3$ into $\lambda$.  We obtain a partition $\mu\in S_2(n)$ such that $\psi(\mu)= (\lambda, (a^2))$.

(iv) If $y \geq 4$ and $c\equiv 1\pmod 5$, since $c$ is odd, we write write $c=10j+1$ with $j \geq 1$. 
  
  If $c=20h+11=4(5h+2)+3$ for some $h\geq 0$, define $$\psi(\mu):=\begin{cases}(\widetilde \mu\setminus \{x\}\cup\{x+3\}, ((5h+2)^4)) & \mbox{ if } h>0, x\neq 0\\ (\widetilde \mu \cup\{3\}, ((5h+2)^4)) & \mbox{ if } h>0, x= 0\\  (\widetilde \mu\setminus \{x\}\cup\{x+2\}, (3^3)) & \mbox{ if } h=0, x\neq 0\\  (\widetilde \mu \cup\{2\}, (3^3)) & \mbox{ if } h=0, x= 0 .\end{cases} $$

  The image under $\psi$ of the subset of overpartitions in $\mathcal S_2(n)$ in this case is \begin{align*} \mathcal{I}_6&(n):= \\&  \left\{(\lambda, (a^4))\in \mathcal S_1(n) \Bigg | \begin{array}{l}  a\equiv 2\pmod 5, a>2 \mbox{ and } \\  4a+3, 4a+4 \not \in \lambda\\ z_{4a+2, \lambda} \geq 5 \mbox{ and } z_{2a+2, \lambda}-t\not \in \lambda \mbox{ for } t=2,3,4\end{array}\right\}\bigcup \\ & \left\{(\lambda, (a^4))\in \mathcal S_1(n) \Big | \begin{array}{l}  a \equiv 2\pmod 5, a>2, \ 3\in \lambda \mbox{ and } \\ \mbox{if }  z \leq 4a+4, z\neq 3,  \mbox{ then } z \not \in  \lambda\end{array} \right\}\bigcup\\ &  \left\{(\lambda, (3^3))\in \mathcal S_1(n) \Big | \begin{array}{l}    10, 11,12 \not \in \lambda\\ \mbox{if } z_{9,\lambda}\geq 4 \mbox{ and }  z_{9,\lambda}-t\not\in \lambda \mbox{ for } t=2,3\end{array}\right\}\bigcup \\ & \left\{(\lambda, (3^3))\in \mathcal S_1(n) \Big | \begin{array}{l}   2\in \lambda \mbox{ and } \\ \mbox{if }  z \leq 12, z\neq 2,  \mbox{ then } z \not \in  \lambda\end{array} \right\}.
   \end{align*} 
   By considering parts less than or equal to $4a+2$ if $a\equiv 2\pmod 5$, and parts less than $10$ otherwise, we see that the four sets whose union is $\mathcal I_6(n)$ are disjoint. As in the previous cases, one can verify that $\psi$ is onto $\mathcal{I}_6(n)$.

  If $c=20h+1$ for some $h\geq 1$, write $c=3m+r$ with $0\leq r\leq 2$. Note that $m \geq 7$. Moreover, if $r=0$, then $m\equiv 7\pmod{20}$; if $r=1$, then $m\equiv 0\pmod{20}$; and if $r=2$, then $m\equiv 13\pmod{20}$. 
 We  define 
 {\small{\begin{align*} \psi & (\mu) := \\&\begin{cases}(\widetilde \mu\setminus \{x\}\cup\{x+r\}, (3^m)) & \mbox{ if } x\neq 0 \\  \widetilde \mu\cup\{r\}, (3^m)) & \mbox{ if } x=0, r\neq 1\\ (\widetilde \mu\cup\{5(h-1)+8, 5(h-1)+6,5(h-1)+4\}, (5(h-1)+3)) & \mbox{ if } x=0, r=1.\end{cases} \end{align*}}}

 The image under $\psi$ of the subset of overpartitions in $\mathcal S_2(n)$ in this case is \begin{align*} \mathcal{I}_7(n)&:= \\ &  \left\{(\lambda, (3^m))\in \mathcal S_1(n) \Bigg | \begin{array}{l}  m\equiv 7\pmod{20} \mbox{ and } \\  3m-3, 3m-2, 3m-1, 3m, 3m+1 \not \in \lambda\\ \mbox{ There is } z\in \lambda, z \leq 3m-4
 \end{array}\right\}\bigcup \\
 &  \left\{(\lambda, (3^m))\in \mathcal S_1(n) \Bigg | \begin{array}{l}  m\equiv 0\pmod{20}, m>0 \mbox{ and } \\   3m+1, 3m+2 \not \in \lambda\\  z_{3m-2,\lambda}\geq 3 \mbox{ and } z_{3m-2,\lambda}-2\not\in \lambda\end{array}\right\}\bigcup \\
 &  \left\{(\lambda, (3^m))\in \mathcal S_1(n) \Bigg | \begin{array}{l}  m\equiv 13\pmod{20} \mbox{ and } \\   3m+1, 3m+2, 3m+3 \not \in \lambda\\  z_{3m, \lambda}\geq 4 \mbox{ and }  z_{3m, \lambda}-t\not\in \lambda \mbox{ for } t=2,3\end{array}\right\}\bigcup \\
& \left\{(\lambda, (3^m))\in \mathcal S_1(n) \Big | \begin{array}{l}  m\equiv 7 \pmod{20} \mbox{ and } \\ \mbox{if }  z \leq 3m+1,  \mbox{ then } z \not \in  \lambda\end{array} \right\}\bigcup \\
& \left\{(\lambda, (3^m))\in \mathcal S_1(n) \Big | \begin{array}{l}  m\equiv 13 \pmod{20}, 2 \in\lambda \mbox{ and } \\ \mbox{if }  z \leq 3m+3, z\neq 2,  \mbox{ then } z \not \in  \lambda\end{array} \right\}\bigcup \\
& \left\{(\lambda, (a))\in \mathcal S_1(n) \Big | \begin{array}{l}  a\equiv 13 \pmod{15}, a+1, a+3, a+5 \in\lambda \mbox{ and } \\ \mbox{if }  z \leq 4a+10, z\neq  a+1, a+3, a+5,  \mbox{ then } z \not \in  \lambda\end{array} \right\}.
 \end{align*} 
 Clearly,  the six sets whose union is $\mathcal I_7(n)$ are disjoint. 
 
 Upon inspection, we see that the sets $\mathcal I_j$, $1\leq j\leq 7$ are mutually disjoint. Their union is the image of $\mathcal S_2(n)$ under $\psi$.   Thus, the excess 
of the number of parts in all partitions of $n$ with   parts congruent to $\pm 2 \pmod{5}$ over the number of parts in all partitions of $n$ with super-distinct parts  {greater than $1$} equals $\left|\mathcal S(n)\right|$, where
\begin{align}\label{def_Snset} \mathcal S(n) := \mathcal S_1(n)\setminus \displaystyle  \bigcup_{j=1}^7\mathcal I_j.\end{align}
 \end{proof}
 
 \begin{example} Let $n=4$. The only partition with parts congruent to $\pm 2 \pmod 5$ is $(2,2)$ and it has two parts. The only partition into super-distinct parts greater than $1$ is $(4)$  and it has one  part. The only   pair of partitions in $\mathcal S(4)$ is $(\lambda, (a^b))= ((2),(2))$. Clearly, $((2), (2))\in \mathcal S_1(4)$.  Since $a=2$ and $b=1$, the pair is not in $\mathcal I_j(4)$ for $j=2,3,5,6,7$. The pair is not in $\mathcal I_1(4)$ because $2b=2$ is a part of $\lambda=(2)$. Moreover, the pair is not in $\mathcal I_2(4)$ because $a+1=3$ is not a part of $\lambda=(2)$.
 \end{example}
 
\begin{remark} The construction of the injection $\psi$ above shows that likely other choices of injections exist. A simpler injection that allows for a nice description of the complement of its image in $\mathcal S_1(n)$ is welcome. 
\end{remark}

\section*{Acknowledgements}  The authors thank the organizers of the Subbarao Centenary Symposium at IISER Pune, July 2021, after which this collaboration began.  The second author is  partially supported by National Science Foundation Grant DMS-1901791.  
\end{document}